\documentclass{article}
\usepackage{authblk}
\usepackage{graphicx} 
\usepackage{babel}
\usepackage{amsmath}
\usepackage{amsfonts}
\usepackage{amsthm}
\usepackage{amssymb}
\usepackage{hyperref}
\usepackage{enumitem}
\usepackage{cases}

\newtheorem{thm}{Theorem}[section]

\newtheorem{cor}[thm]{Corollary}

\theoremstyle{definition}
\newtheorem{defn}[thm]{Definition}

\title{On a Digital Signature Based on Matrix Decomposition and the MinRank Problem}
\author{Miguel Ángel López\thanks{migulo23@ucm.es}}
\affil{Complutense University of Madrid}
\date{}

\begin{document}

\maketitle

\begin{abstract}
$x\cdot 3$ The following article attempts to define a new digital signature scheme based on a variant of the MinRank problem under appropriate parameters. First, the key generation process will be defined, followed by the message signing and verification processes. The speed of signing and verification will be analyzed, and finally, the signature will be evaluated in terms of security and robustness against known attacks, which will lead to a set of requirements that we will consider when selecting certain parameters related to the size and number of variables. It should be noted that this signature, although simple in design, can serve as the foundation for an entire family of digital signatures based on the central idea that inspired its creation.
\end{abstract}

\section{Digital Signature Key Generation}

From now on, we will assume that we are working, as is customary in crypto\-gra\-phy, in a finite field $\mathbb{F}_q$. These fields can be of the form $\mathbb{Z}_p$, where $p$ is a prime number, or have a number of elements equal to $q=p^k$, where $k\in\mathbb{N}$.

Let $m$ and $n$ be two integers in $\mathbb{N}$. We choose two $n \times n$ matrices $A$ and $B$ over $\mathbb{F}_q$ that satisfy the following conditions:

\begin{enumerate}
\item $A$ and $B$ are circulant matrices, that is, matrices in which each row is the next one shifted cyclically by one to the right. In our case, they correspond to the general form
\begin{equation}
    \begin{pmatrix}
    a_0 & a_1 & \ldots & a_{n-1} \\ 
    a_{n-1} & a_0 & \ldots & a_{n-2}  \\
    \vdots & & \ddots & \vdots \\ 
    a_1 & a_2 & \ldots & a_0 
    \end{pmatrix}
\end{equation}
\item All elements in the first row of $A$ and $B$ (and therefore in both matrices) are distinct polynomials of degree $1$ in the variables $x_1\ldots x_m$, that is, $a_i=a_i (x_1\ldots x_m)\in\mathbb{F}_q[x_1\ldots x_m]$, where $deg(a_i)=1$, $1\leq i\leq n$.
\end{enumerate}

We construct $n$ matrices $S_i$ of size $n\times n$ over $\mathbb{F}_q$ as follows:

\begin{enumerate}
\item We generate $n$ circulant matrices $\tilde{S_i}$ of size $n\times n$.
\item We choose $p$ rows between $1$ and $n-2$ and replace, in all the $\tilde{S_i}$ matrices, the rows at the positions of the $p$ values with new rows containing values chosen at random over $\mathbb{F}_q$. Each new row must be different for each $\tilde{S_i}$, and thus we have generated $p\cdot n$ new rows in total. We will call these rows “intruder rows” and rename the matrices as $\hat{S_i}$.
\item We impose the condition that each of these $\hat{S_i}$ has a nonzero determinant; otherwise, we modify the matrices $\hat{S_i}$ that do not satisfy this condition with new intruder rows until they do. It must also hold that the $n\times n$ matrix $C$, formed by the first non-intrusive row of each $\hat{S_i}$, which we will call the circulating row, must also have a nonzero determinant. Once these two conditions are satisfied, we rename all the matrices as $S_i$, $1\leq i\leq n$.
\end{enumerate}

Each of these matrices satisfies the condition that the submatrix resulting from removing the intruding rows contains rows created by cyclically shifting the circulant row, but none of the $S_i$ matrices is actually a circulant due to the intruding rows that have been introduced.

\textbf{Public key:} $P_i(x_1\ldots x_m)=A\cdot S_i\cdot B$, $1\leq i\leq n$. It consists of $n$ $n\times n$ matrices whose entries are quadratic polynomials in the variables $x_1\ldots x_m$. We also publicly declare the total number $p$ of intruder rows. This value will be the rank used to verify the signature, and when it is more clear, we will use the notation $r=p$.

The \textbf{length of the public key} is $$\frac{n^3(m+1)(m+2)}{2}$$ and depends on the size $n$ chosen for the matrices and the number $m$ of variables. This value will be the range used to verify the signature, and when it is more clear, we will use the notation $r=p$.

\textbf{Private key:} Although for the sake of brevity we can say that the private key is $A,B,S_i$, where $1\leq i\leq n$, in reality it is sufficient to store the two matrices $A$ and $B$ and a single matrix $S$ consisting of the first non-intrusive row of each matrix $S_i$, where $1\leq i\leq n$.

This greatly reduces the \textbf{length of the private key}, which is $$n(2m+n+2),$$

much shorter than that of the public key.

\section{Signing Process}

Given a message $M$, we use auxiliary functions (which include hash functions and padding) to convert it into a fixed-length vector $\vec{d}\in\mathbb{F}_q$. From this, we generate a circulant matrix $D$ whose first row is $\vec{d}$. If $\left | D \right |=0$, we modify $\vec{d}$ as many times as necessary until $\left | D \right |\neq 0$.

We choose $x_1\ldots x_m$ at random and substitute them into $A$ and $B$ to obtain $\tilde{A}$ and $\tilde{B}$, which are $n\times n$ matrices over $\mathbb{F}_q$. If $\left | \tilde{A} \right |=0$ or $\left | \tilde{B} \right |=0$, we choose other variables.

We compute $\tilde{D}=\tilde{A}^{-1}\cdot D\cdot\tilde{B}^{-1}$, which will be circulant and full-rank by construction.

We solve the MinRank problem by considering the matrices $-\tilde{D}$ and the matrices $S_i$, where $1\leq i\leq n$. If we denote the position of the circulant row of each $S_i$ as $j_0$, this is equivalent to finding the solution to the following linear system:

\begin{equation}
    \left(
    \begin{array}{ccc|c}
    s_{j_0,1}^{(1)} & \ldots & s_{j_0,1}^{(n)} & -\tilde{d}_{j_0,1}\\
    \vdots & & \vdots & \vdots \\ 
    s_{j_0,n}^{(1)} & \ldots & s_{j_0,n}^{(n)} & -\tilde{d}_{j_0,n}
    \end{array}
    \right)
\end{equation}

Due to the third condition we imposed on the matrices $S_i$, this system is compatible determined and therefore has a unique solution $\vec{g}\in\mathbb{F}_q^n$.

The signature consists of the message $M$, the functions needed to construct $\vec{d}$, the variables $x_1\ldots x_m$, and the vector $\vec{g}\in\mathbb{F}_q$.

\section{Signature Verification}

$M$ is transformed into $\vec{d}$ using the submitted functions, and $D$ is constructed as in the signing process. The matrices $$\hat{P}_i=P_i(x_1\ldots x_m)$$ are computed for $1\leq i\leq n$. If the determinant of any of these matrices is zero, the signature is rejected. The matrix $$G=D+\sum_{i=1}^n g_i\cdot\hat{P}_i$$ is computed, and the rank of $G$ is then calculated. If it is equal to or less than $r$, the signature is accepted. Otherwise, it is rejected.

It is worth noting that the main premise underpinning the usefulness of the signature lies in the fact that the generalized MinRank problem is NP-complete when defined over a finite field, as can be seen in \cite{buss1999computational}.

\section{Speed of the Signing and Verification Process}

Given that having to repeat the selection of the vector $\vec{d}$ and the variables $x_1\ldots x_m$ does not significantly affect the number of steps in the algorithm, the complexity would be as follows:

\begin{itemize}
    \item Step 1, generating $D$ from $\vec{d}$, has a complexity of $O(n^2)$.
    \item Choosing $m$ variables at random has a complexity of $O(m)$. Substituting them has a complexity of $O(n\cdot m+n^2)$. Verifying their suitability (that the determinants of $\tilde{A}$ and $\tilde{B}$ are nonzero, and repeating the process if they are not) will not add any additional complexity to the process, since, as we will see later, we will choose the parameters so that the probability of them being valid after a small number of iterations is high. From this point on, the complexity no longer depends on $m$.
    \item The complexity of computing $\tilde{D}=\tilde{A}^{-1}\cdot D\cdot\tilde{B}^{-1}$ is $(2n^w+2n^w)$, since it involves two inversions and two products of $n\times n$ matrices, and is therefore $O(n^w)$, where $w$ is the constant from linear algebra, $2<w<3$.
    The next step is to solve an $n\times n$ linear system, which we know is consistent by construction and has a runtime of $O(n^w)$. This allows us to find the vector $\vec{g}\in\mathbb{F}_q^n$.
\end{itemize}

The complexity is therefore $O(\max(n^w, n \cdot m + n^2))$. In practical examples, we will aim for $m \leq n$, so with the chosen parameters, our actual complexity for the signing process will be $O(n^w)$.

For signature verification, we need to evaluate all the matrices $P_i(x_1\ldots x_m)$, for a total of $n$ matrices of size $n\times n$. In total, there are $n^3$ elements, where each element is a quadratic polynomial in $m$ variables. Evaluating each quadratic polynomial is equivalent to multiplying a $1\times (m+1)$ row matrix by a $(m+1)\times (m+1)$ square matrix by a $(m+1)\times 1$ column matrix. The complexity of this product is equal to $$O((m+1)^2+m+1)=O(m^2+2m+1+m+1)=O(m^2)$$

The complexity of the evaluation is $O(m^2 n^3)$, and from this point on, it no longer depends on $m$.

The complexity of determining whether an $n\times n$ matrix is full-rank is $O(n^w)$, and since we must check this for $n$ matrices, the total complexity is $O(n^{w+1})$.

The computation of $G$ has a complexity of $O(n^3)$ for performing all the multiplications and $O(n^3)$ for summing the $n^2$ values a total of $n$ times. The final sum over the matrix $D$ can be considered constant for the complexity, and therefore the overall complexity for computing $G$ is $O(n^3)$.

To find the rank of $G$, we perform standard Gaussian elimination, which has a complexity of $O(n^w)$.

The complexity of the signature verification, therefore, will be $$O(\max (m^2 n^3, n^{w+1}))$$ For practical purposes, we will aim for $m \leq n$, so with the chosen parameters, our actual complexity will be $O(\max(n^5, n^{w+1}))$.

Both processes, signing and verification, will therefore be polynomial in the worst case, even if we cannot restrict the values of $m$.

\section{Parameter Selection and Security Analysis of the Digital Signature}

From here on, we will assume that for verification we require a range less than or equal to $r$ and, therefore, we have designed the signature by introducing exactly $p=r$ intruder rows. We want to fix the values of $r, q$ (the size of the finite field), $n$ (the number and size of the matrices $S_i$), and $m$ (the number of variables $x_j$) to obtain a signature that meets certain security standards.

\subsection{Swapping Rows and Columns Doesn't Enhance Se\-cu\-ri\-ty}

It is possible to design the digital signature without having to publicly declare whether the design was based on rows or columns. In the event of a brute-force attack, however, this would only double the number of possible combinations of different choices, so the complexity of any attack, if it is $O(f(n))$, would become $O(2f(n))$, which does not strengthen security.

\subsection{Avoiding Variable Collisions}

If the same variables $x_1\ldots x_m$ are used twice, the signature acquires linear properties. The probability that the signature owner, when choosing two sets of variables at random, will select the same set of variables is $\frac{1}{q^m}$. We must therefore choose $q$ and $m$ such that this probability is as small as we wish.

\subsection{Avoiding Diagonalization of Matrices $A$ and $B$}

The matrices can be diagonalized under certain conditions using a discrete Fourier transform of size $n$ that is of the form

\begin{equation}
DFT(n)=
    \begin{pmatrix}
       1 & 1 & \ldots & 1 \\
       1 & \alpha & \ldots & \alpha^{n-1} \\
       \vdots & \vdots & \ddots & \vdots \\
       1 & \alpha^{n-1} & \ldots & \alpha^{(n-1)(n-1)}
    \end{pmatrix}
\end{equation}

where $\alpha$ is a primitive $n$th root of unity in $\mathbb{F}_q$. If this transformation exists, it is generally not unique.

We do not want to allow the possible use of this tool in $\mathbb{F}_q$, so we will prevent the existence of $n$th-root of unity in this finite field. For $n$th-roots of unity to exist, it must be the case that $n | (q-1)$, as explained in \cite{petersen2011use}. There would in fact be $n$ such roots, and it would suffice to choose a primitive root of unity (which in fact exists if and only if $n|(q-1)$ and, moreover, there are $\varphi(n)$ of them in total) to generate a $DFT(n)$. For example, for $n=3$ and $q=7$, there are three cubic roots of unity in $\mathbb{Z}_7$, which are $1$, $2$, and $4$. Both $2$ and $4$ are primitive, and if we choose, for example, the first one, $DFT(n)=\begin{pmatrix}
       1 & 1 & 1 \\
       1 & 2 & 4 \\
       1 & 4 & 2       
    \end{pmatrix}$. Therefore, we impose the condition that $n\not|(q-1)$.

\subsection{Choosing the parameter $r$}

If we choose to add $p$ intruder rows, the total number of possible positions for these rows is equal to ${n}\choose{p}$, which reaches its maximum when $p=\lceil{\frac{p}{2}}\rceil$. As we have already noted, considering columns instead of rows is irrelevant, since it would only double the possibilities, which is negligible in terms of attack complexity. To attack this signature using brute force or methods that require assuming the initial structure of the $S_i$, an attacker must multiply their computational effort by a combinatorial factor a total of ${n}\choose{p}$ times, and this is indeed significant in terms of complexity.

The maximum of these values occurs in the so-called central binomial coe\-ffi\-cient, which is equal to ${2n}\choose{n}$, but since we also want to consider it for odd rows in general, we will write it as $${{n}\choose{\lceil{\frac{n}{2}}\rceil}}, n \in \mathbb{N}$$

It can be verified that

$${n\choose\lceil{\frac{n}{2}}\rceil}=O\left( \frac{4^n}{\sqrt{n}}\right),$$

which accurately reflects the added difficulty of having to test the positions of all these new intruding rows, multiplied by a growth rate that is practically exponential in $n$. Specifically, according to \cite{luke1969special}, the binomial coefficient $2n\choose n$ exhibits the following asymptotic behavior:

$${2n\choose n}=\frac{4^n}{\sqrt{\pi n}}\left(1-\frac{1}{8n}+\frac{1}{128n^2}+\frac{5}{1024n^3}+O\left(n^{-4}\right)\right)$$

These numbers are related to Catalan numbers through the follo\-wing for\-mu\-la:

$$C_n=\frac{1}{n+1}{2n\choose n}$$

The asymptotic bound in the case of ${n\choose{\lceil{\frac{n}{2}\rceil}}}$, where $n$ is odd, is analogous, given that

$${{2n-1}\choose{\lceil{\frac{2n-1}{2}\rceil}}}=\frac{1}{2}{2n\choose n}$$

In terms of robustness, therefore, the closer we get to $r=p=\lceil{\frac{p}{2}}\rceil$, the more operations will be required to attack the signature using brute force or methods derived from the internal structure of the $S_i$, but as we will see later, unfortunately this value cannot be achieved exactly, and we will have to settle for getting as close as possible given certain additional conditions.

\subsection{Consider a well-determined MinRank problem}

The MinRank problem can be under-determined, well-determined, or over-determined. Of these, the well-determined problem is the most suitable for our firm for several reasons. In the case of an under-determined problem, \cite{faugere2010computing} explains that this class of problem can always be reduced to a well-determined one for theoretical purposes, so we gain nothing in terms of additional security. Furthermore, if we underdetermine the problem, for example, by altering the number of matrices, our linear system could become inconsistent; in fact, this could happen even if we forced it to be well-determined. The fact that the associated MinRank has a single solution does not necessarily imply that it must be the solution to our linear system.

On the other hand, the over-determined case largely lacks theoretical results regarding the attack on the problem using standard methods, which are Minors Modeling, KS Modeling, and Linear Algebra Search. In \cite{caminata2021complexity}, upper bounds on the solving degree are discussed for Minors Modeling in the under-determined and well-determined cases, but it is specified that analogous results are lacking for the over-determined case.

We therefore want to ensure that the problem is well-determined in order to minimize the chances that someone will create a different signature for the same set of variables, but we cannot do so by altering the number of matrices. We will therefore do so using the value $r$.

Ensuring that the problem is well-determined is not straightforward, but the paper \cite{courtois2001efficient} demonstrates (p. 15, A. Probability Distribution of Ranks) that for a given MinRank problem involving $m$ $\eta\times n$ matrices with a required rank of $r$, the closer $m$ is to the value $m_{max}$ where

$$m_{max}=\eta n+r^2-(\eta+n)r+1,$$

the more likely it is, on average, that there exists a solution and exactly one solution. In our case, we have the constraints that $\eta = n$ (i.e., the matrices are square) and that $m = n$ (the article defines the MinRank problem in a homogeneous manner and therefore considers one fewer matrix), and solving for $r$ yields $r = n - \sqrt{n - 1}.$

If we try to impose the condition that $r=\lceil{\frac{n}{2}}\rceil$ at the same time, we find that we have to solve the system

$$\begin{cases}
    r=n-\sqrt{n-1} \\
    r=\lceil{\frac{n}{2}}\rceil
\end{cases}
$$

which yields as the only natural solutions

$$(n,r)\in\left\{(1,1),(2,1),(5,3)\right\},$$

which are either nonsensical solutions, given that $n$ is the size of a square matrix, or are too small to be considered useful in practice.

There are two ways to reconcile these two requirements for the rank. One is to modify the number $k$ of matrices so that the equation

$$k=\eta n+r^2+(\eta+n)r+1,$$

is satisfied, given that $r=\lceil\frac{n}{2}\rceil$ and $\eta=n$, in which case it would be

$$k=(n-r)^2+1=\begin{cases}
    \left(\frac{n}{2}\right)^2+1 & \text{if $n$ even} \\
    \left(\frac{(n-1)}{2}\right)^2+1 & \text{if $n$ odd}
\end{cases}$$

We have already stated that we had no intention of making such a modifica\-tion, and here you can see why this would cause several problems. If, for example, $n=5$, then $k=5$ and $r=3$, which would leave us with three extra rows for a signature on five $5\times5$ matrices. This would mean that during the signing process, we would have to solve a linear system of five equations with four unknowns (the fifth matrix serves to define the right-hand side of the system), which in the vast majority of cases would be inconsistent. On average, there would be a single solution, but it would not be the one we construct during the signing process; it would be another one that we would not be able to find.

If, on the other hand, we take $n=20$, we have $k=101$ and $r=10$, and with our signature design, the linear system to be solved would be consistent but indeterminate with many degrees of freedom, leading us to conclude that the matrices we would have designed are not in the mean.

It is therefore not advisable to change the number of matrices, due to the risk that we may not be able to solve the linear system or that it may have an excessive number of solutions—the existence of which, moreover, offers us no additional benefit, since to sign the same document twice we would choose different random variables and, therefore, would have to solve another linear system rather than use the same one again. The most appropriate approach is that, given a set of variables, the signature be unique to them, since we will require that there be no collision among the variables, thereby preventing, as much as possible, the same document—with the same variables—from being forged with a different signature.

If, therefore, we assume that we cannot solve the system $$\begin{cases}
    r=n-\sqrt{n-1} \\
    r=\lceil{\frac{n}{2}}\rceil
\end{cases}$$ we work under the assumption that the equation we cannot ignore is the first one. Our new “intruding” terms are therefore $p=r=n-\sqrt{n-1}$, and there will be $${n\choose n-\sqrt{n-1}}={n\choose \sqrt{n-1}}$$ different ways to order them.

It can be easily verified experimentally that, if we define the function

$$f(n)={n\choose \lceil{\frac{n}{2}\rceil}}-{n\choose\sqrt{n-1}},$$

we find that $f(4)>0$ and $f$ is increasing; in fact, as the value of $n$ increases, the two candidate values for $r$ become further apart. The gap between these two values would therefore be smaller as the value of $n$ decreases, which we cannot assume since the value of $n$ must be large enough for the associated MinRank problem to be computationally intractable.

Assuming, therefore, that $r=n-\sqrt{n-1}$, we will calculate how many possibilities we add asymptotically as $n$ increases for the $p=r$ intruding rows thus added.

We have that $${n \choose\sqrt{n-1}}\thicksim\frac{(n+1)^{\sqrt{n-1}}}{\sqrt{n-1}!}$$ since for values of $n$ and $k$ that are sufficiently far apart, we have that ${n-1\choose k-1}$ is approximately $\frac{n^{k-1}}{(k-1)!} $. Making the substitution $u^2 = n - 1$, we find that the trend of the first expression is equivalent to considering that of $\frac{(u^2 + 2)^u}{u!}$, and to handle this expression we can use Stirling’s approximation, $u!\thicksim\sqrt{2\pi u}\left(\frac{u}{e}\right)^u$, which is equivalent to $$\frac{e^u}{\sqrt{2\pi u}}\left(u+\frac{2}{u}\right)^u.$$

This expression is of the order $O(e^{2u}(u-1)!)$ with $u=\sqrt{n-1}$, so for large values of $n$, we continue to add an unmanageable number of possibilities to brute-force attacks that rely on the internal structure of the $S_i$.

\subsection{Preventing a Direct Matrix Factorization Attack}

A potential attacker has access to the public key matrices, that is, the $P_i=A\cdot S_i\cdot B$, where $1\leq i\leq n$. The attacker wishes to obtain the matrices $A$, $S_i$, and $B$. Each $P_i$ has $n^2$ elements, and there are $n$ of them in total. Each element is a multivariate quadratic polynomial in $m$ variables; there are at most $\frac{(m+1)(m+2)}{2}$ nonzero coefficients, but to obtain the original values, even these must be equated. There are therefore $$e(n,m)=\frac{n^3(m+1)(m+2)}{2}$$ equations to solve; for the following discussion, we will not consider the equations over a finite field, since the inequalities mentioned below do not take them into account.

The unknowns are $n(r+1)$ values that define each $S_i$, the $n(m+1)$ values in the first row of $A$, and the same number of values for $B$; therefore, the total number of unknowns is equal to $$v(n,m,r) = n(n + rn + 2m + 2). $$ We also assume that $r=n-\sqrt{n-1}$, so we have that $$v(n,m)=n(n^2+n-n\sqrt{n-1}+2m+2).$$

In \cite{ding2006multivariate}, it is explained that, given a multivariate quadratic system with $v$ variables and $e$ equations, there exist polynomial-time algorithms for certain special cases, namely $v\geq \frac{e(e+3)}{2}$, $e\geq \frac{v(v+3)}{2}$, or $v=\lambda e$, where $\lambda\in\mathbb{N}$, and these correspond to systems in which the difference between the number of equations and the number of unknowns is very pronounced. These include algorithms such as Relinearization or the XL algorithm. The system we are solving is not quadratic (it is a non-homogeneous system of degree 3, but it can be converted to a homogeneous system without loss of generality, as considered in \cite{courtoisXL}); however, we will still impose these bounds to prevent future attacks that exploit the fact that the system might be too unbalanced between equations and variables.

It can be easily verified that for $m\in(4,\infty)$, all natural values of $n$ satisfy $e(n,m)>v(n,m)$. We can therefore ignore the cases $v\geq \frac{e(e+3)}{2}$, $v=\lambda e$, where $\lambda\in\mathbb{N}$. We are left to analyze the case $e\geq \frac{v(v+3)}{2}$, and to do so, we will derive the inequality associated with $e<\frac{v^2}{2}$, which is $$n^2(n^2+n-n\sqrt{n-1}+2m+2)^2-n^3(m+1)(m+2)>0.$$

This inequality is very lenient; for example, if $m=20$, any value of $n$ is acceptable, and the first value of $m$ for which not just any value of $n$ is acceptable is $m=58$, which requires $n\geq 8$.

\subsection{Avoiding a matrix factorization attack after choosing the variables}

A simpler approach to decomposing matrices is not to do so for all variables, but rather after selecting a group of them; that is, given matrices $P_i\in M_n (\mathbb{F}_q)$ of the form $P_i=A\cdot S_i\cdot B$, obtain the original matrices $A, S_i, B$, and thereby impersonate the owner of the signature for a specific group of variables, which would allow a message to be forged.

This attack is independent of the number of variables $m$, and, in theory, a variable-based attack would be more computationally expensive. It has been verified using Magma that even small values for matrices and rank, such as $n=15$ and $r=1$ (assuming the intruder row is in the first position), lead to systems of unmanageable sizes when solved using Gröbner bases and Faugère’s F4 algorithm. In general terms, the number of equations to be solved is $e(n)=n^3$, and the number of unknowns is $v(n,r)=2n+(r+1)n^2$. Again, we assume $r=n-\sqrt{n-1}$.

To rule out the case $v=\lambda e$, where $\lambda\in\mathbb{N}$, we must solve the Diophantine equation in two non-negative variables $$(1-\lambda)^2n^4+(1-2\lambda)n^3+(6-4\lambda)n^2+4n+4=0. $$ A direct inspection shows that if $\lambda \geq 5$, this equation has no real solutions, and if $0 < \lambda < 5$, all its real roots are less than $4$ and, furthermore, none of them is an integer.

If we impose $v < \frac{e(e+3)}{2}$, we see that this holds for all $n > 2$. The last case, $e<\frac{v(v+3)}{2}$, holds for every natural number $n$. On the other hand, the multivariable system that would result and need to be solved is a non-homogeneous system of degree 3 and is therefore generally intractable for large values of $n$.

\subsection{Ensuring the Existence of Sufficient Circulant and In\-ver\-ti\-ble Matrices in a Finite Field}

When generating digital signature keys and substituting the variables $x_1\cdots x_m$, we need to find circulant matrices of size $n$ over our field $\mathbb{F}_q$ that are invertible in just a few attempts. The total number of circulant matrices of size $n$ over $\mathbb{F}_q$ is $q^n$. How many of these are invertible?

Since every circulant matrix of size $n$ over $\mathbb{F}_q$ has an associated polynomial of degree $n-1$, and since circulant matrices form a ring isomorphic to $R\equiv\mathbb{F}_q\left[x\right]/\left\langle x^n-1\right\rangle$, it suffices to count the invertible ones over $R$. When $q$ and $n$ are coprime, $R$ is a direct product of fields, specifically $R\simeq F_1\times\ldots\times F_r$, where $F_j\simeq\mathbb{F}_q\left[x\right]/\left\langle f_j(x)\right\rangle$, where $f_1\ldots f_r$ are the irreducible factors of $x^n-1$ over $\mathbb{F}_q$. Then $R$ has $\prod^r_{j=1}\left( q^{d_j}-1\right)$ invertible elements, where $d_j$ is the degree of $f_j$.

This implies that the fewer factors $x^n-1$ has over $\mathbb{F}_q$, the fewer zero divisors $R$ will have and the more invertibles it will contain. Therefore, we must choose $n$ and $q$ such that this factorization is not too large. We must therefore study the following problem: given $\mathbb{F}_q$, a finite field of characteristic $p$, for what values of $n$ does the polynomial $x^n-1$ factor into the smallest possible number of factors? We know that this number of factors is at least $2$, since $x^n - 1 = (x - 1) \sum_{j=0}^{n-1} x_j$ in any polynomial ring $A[x]$.

To study this, we will need the concept of a cyclotomic polynomial in finite fields, which is simply a polynomial of fixed degree $n$ whose roots are the $n$th roots of unity in $\mathbb{F}_q$. These concepts are discussed in greater depth in \cite{Huczynska2010FiniteFields}.

\begin{defn}
    Given a finite field $\mathbb{F}_q$ of characteristic $p$ (where $p=0$ is allowed), given $n$ a positive integer not divisible by $p$, and $\xi$ an $n$th root of unity in $\mathbb{F}_q$, we define $$Q_n(x)=\prod_{s=1}^n\left (x-\xi^s\right)$$ as the $n$-cyclotomic polynomial over $\mathbb{F}_q$, which has degree equal to $\varphi(n)$, where $\varphi$ is Euler’s function.
\end{defn}

It can be shown that $x^n-1=\prod_{d \mid n}Q_d(x)$, and that all coefficients of $Q_n(x)$ belong to the prime subfield of $\mathbb{F}_q$, which is isomorphic to $\mathbb{Z}_p$ ($\mathbb{Z}$ in the case of the field $\mathbb{Q}$ of characteristic zero). This therefore implies that $$Q_n(x)=\frac{x^n-1}{\prod_{\substack{d \mid n \\ d<n}}Q_d(x)}$$

and in particular, we have that $Q_r(x)=1+x+x^2+\ldots x^{r-1}$ if $r$ is prime. We therefore have that, if $n$ is composite, then it has at least one prime divisor $u$ and therefore $Q_u(x)\mid \left(x^n-1\right)$, where $Q_u(x)=\sum_{j=1}^{u-1}u^j$ is a proper divisor of the ring of polynomials in one variable over the finite field $\mathbb{F}_q$. Therefore, for $x^n-1$ to factor into only two factors over $\mathbb{F}_q$, it is a necessary condition that $n$ be a prime number. This, however, is not a sufficient condition, since we also need $Q_n(x)$ to be irreducible over $\mathbb{F}_q [x]$, and this will not hold for every prime $n$.

At this point, we will also need the concept of order.

\begin{defn}
    Let $n$ be a positive integer, and $b$ an integer coprime to $n$. The smallest positive integer $k$ such that $n\mid(b^k-1),$ which is equivalent to $b^k\equiv 1\pmod{n},$ is called the multiplicative order of $b$ modulo $n$ and is denoted by $ord_n(b)=k$.
\end{defn}

We will use a very important theorem that states the following.

\begin{thm}
    Given a finite field $\mathbb{F}_q$, a positive integer $n$ such that $\gcd(q,n)=1$, and $d=ord_n(q)$, then $Q_n(x)$ factors into exactly $\varphi(n)/d$ distinct polynomials in $\mathbb{F}_q[x]$, all of degree $d$.
\end{thm}

Thanks to this theorem, we can prove the following:

\begin{cor}
    Let $\mathbb{F}_q$ be a finite field, $n$ a prime distinct from $p$ such that $q = p^k$ and $ord_n(q) = n - 1$. Then $Q_n(x)$ is irreducible over $\mathbb{F}_q[x]$.
\end{cor}

\begin{proof}
    Since $n$ is prime and distinct from $p$, then $\gcd(n,q)=1$. By the previous theorem, $Q_n(x)$ factors into $\varphi(n)/ord_n(q)$ polynomials, but $\varphi(n)=n-1$ and $ord_n(q)=n-1$ by assumption, so $\varphi(n)/ord_n(q)=1$ and $Q_n(x)$ is irreducible over $\mathbb{F}_q[x]$.
\end{proof}

This implies, therefore, that for $x^n-1$ to factor into only two factors over $\mathbb{F}_q[x]$, it is a necessary condition that $n$ be a prime number distinct from the characteristic of the finite field and that, furthermore, $ord_n(q)=n-1$.

\begin{cor}
    It is also a sufficient condition; that is, given a finite field $\mathbb{F}_q$ and $n > 2$ such that $x^n - 1 = (x - 1) \sum_{j=1}^{n-1} x^j$ and has no other factors, then $n$ is prime, $\gcd(n, q) = 1$, and $ord_n(q) = n - 1$.
\end{cor}

\begin{proof}
    We have already shown that if $n$ is composite, then $x^n - 1$ has more than two factors; therefore, if $x^n - 1$ has exactly two factors, $n$ is prime.

    If we assume that $\gcd(n,q) \not=  1$, then $\gcd(n,q)=n$ (since $n$ is prime); thus, $q=n^k$ and the characteristic of $\mathbb{F}_q$ is $n$. This implies that $x^n-1=(x-1)^n$, which always has more than two factors if $n>2$, which cannot be the case. Therefore, $\gcd(n,q)=1$.

    Finally, suppose that $ord_n(q) \neq n-1$ and let $ord_n(q) = d$. Since $n$ is prime, the multiplicative group of $\mathbb{Z}_n$ has $n-1$ elements, and by Lagrange’s theorem, the order $d$ must divide $n-1$. Since $n$ is prime, $Q_n(x)=\sum_{j=1}^{n-1}x^j$, and we know that it factors into $\varphi(n)/d$ distinct factors; since $d<n-1$, $\varphi(n)/d=(n-1)/d\geq 2$. Therefore, $x^n-1$ has at least three factors, which is impossible by assumption. Thus, $ord_n(q) = n-1$.
\end{proof}

\begin{cor}
    In a finite field of order $q$, if for our signature we choose a ma\-trix size $n$ such that $n > 2$ is prime, $\gcd(n, q) = 1$, and $ord_n(q) = n - 1$, the probability $P$ that a circulant matrix is invertible is
    
    $$P = \frac {(q^1-1)(q^{n-1}-1)}{q^n}=\frac{q-1}{q}-\frac{q-1}{q^n},$$
    
    so for enough large values of $n$ that satisfy the required conditions, this probability will be less than $(q-1)/q$, an efficient value for digital signatures.
\end{cor}

For example, if $\mathbb{F}_q=\mathbb{Z}_2$, $n=19$, we have that $ord_{19}(2)=18$, so $x^{19}-1$ factors into two factors, which gives us a probability $P$ of obtaining an invertible circulant matrix of size $19$ over $\mathbb{Z}_2$ equal to $p\simeq 0.499998092$. Using the geometric distribution, we have that $$P[\text{invertible on the third attempt or earlier}] = p + (1 - p)p + (1 - p)^2p \simeq 0.874999,$$

which is a more than reasonable number of attempts to quickly obtain an invertible circulant matrix. In fact, the average number of total attempts required to achieve the first success is given by $E[X]=\frac{1}{p}\simeq 2.000007629$, where $X$ is the random variable that defines this geometric distribution.

\subsection{Preventing the possibility of existential forgery}\label{exfor}

Now, as attackers, we want to try to sign any message, regardless of whether its content is nonsense. We therefore want to fraudulently generate a pair $(M, f)$ consisting of a message $M$ and an associated signature $f = (f_1, \ldots, f_n)$. To do this, we only have the public matrices $P_i(x_1\ldots x_m),$ $1\leq i\leq n$. For now, let’s assume that we are free to choose the circulant matrix $\tilde{D}$ on which we will compute the MinRank (this will not actually be the case, but for now we will ignore that situation).

We therefore have the matrix $E=\left[\sum_{i=1}^{n}f_iP_i(x_1\ldots x_m)\right]-\tilde{D}$ and we want to ensure that the rank of $E$ is $r\leq n-\sqrt{n-1}$. To do this, we can first fix $f=(f_1\ldots f_n)$, the variables $x_1\ldots x_m$, or the values of the first row of $\tilde{D}$, which are $d_1\ldots d_n$. In terms of equations, the method of minors tells us that this is equivalent to solving $(n-r)^2=n-1$ multivariate equations with a maximum degree of $4$ if we do not fix any variables beforehand. In this case, the number of variables would be $v(n,m)=2n+m$, and the number of equations, independent of $m$, would be $e(n)=n-1$.

To avoid the cases previously considered involving an excess of variables or equations, we must have that $$e<\frac{v^2}{2}\forall n,m\in\mathbb{N}\backslash\{0\}. $$ If we impose the constraint $$v<\frac{e(e+3)}{2}$$ it immediately follows that $$m<\frac{n^2-3n-2}{2}$$ on the other hand, $v\neq\lambda e$ if $(n-1)\nmid (2n+m)$. These will therefore be the conditions we take into account when choosing the variables.

However, we can partially fix some of these variables. The number of equations remains the same, but $v(n,m)\in\{n,m,2n,m+n\}$ depending on how we fix $$\{(x_j,f_i),(x_j,d_i),(f_j,d_i),(x_j),(f_i),(d_i)\}.$$ The cases $\{(x_j,f_i),(x_j),(f_i)\}$ are ruled out since they imply that we would obtain values to reconstruct a matrix $\tilde{D}$, but we would not be able to reverse this matrix to obtain a message associated with it, given that hash functions and padding will be applied to the message $M$ to arrive at $\tilde{D}$, making it impossible to reverse the process.

This leaves the possibility of specifying $\{(x_j,d_i),(f_j,d_i),(d_j)\}$. The first case is not considered because it is, in itself, the signing process (choosing variables, a message, and obtaining a signature). The second case is very interesting, as it plays with the idea of fixing the signature and the matrix, and then obtaining matching variables that fit into the signature verification process. In this case, the system would have $e(n)=n-1$ equations and $v'(n,m)=m$ variables. It would also be quadratic, so special attention must be paid to ensure that the system is balanced between equations and variables.

First, we have that $$m=v'(n,m)<v(n,m)=2n+m,$$ so immediately $$v'(n,m)<\frac{e(e+3)}{2}$$ if the condition we previously specified is satisfied, $$m<\frac{n^2-3n-2}{2}.$$ The case $e<\frac{v'(v'+3)}{2}$ is equivalent to solving the inequality $$m^2+3m-(2n-2)>0, m, n\in\mathbb{N}$$ and we find that $$m>\frac{-3\pm\sqrt{8n+1}}{2}$$ must hold. Finally, $v' \neq \lambda e$ if $(n-1) \nmid m$.

If we fix only the entries of the matrix $\tilde{D}$, the system is of degree $3$ in the variables $d_1\ldots d_n$. Again, we have that

$$n+m=v''(n,m)<v(n,m)$$

and therefore

$$v’' < v < \frac{e(e+3)}{2}$$

if

$$m<\frac{n^2-3n-2}{2}.$$

For the case $$e<\frac{v''(v’'+3)} {2}$$ we must have $$n-1<\frac{(n+m)(n+m+3)}{2},$$ which is equivalent to $$(n+m)^2+3m+n+2>0,$$ which is true for all $n,m\in\mathbb{N}\backslash\{ 0\}$. Finally, $v'' \neq \lambda e$ if $(n-1) \nmid (n+m)$.

\subsection{Other Security Parameters of a Digital Signature}

In \cite{boneh2020graduate}, other security parameters for a digital signature are defined. On page 529, a game involving the signature is defined; if the attacker succeeds in this game, it constitutes an attack via existential forgery. If the signature withstands an attack via existential forgery, it is said to be secure, which would be the case here, provided the parameters are chosen appropriately, as discussed in the previous section \ref{exfor}.

There is a higher level of security, called a strongly secure signature, which means that an adversary cannot even sign messages that have already been signed using a new signature. This level of security can, in fact, be achieved, since the same source states that any signature scheme that is secure can be converted into a strongly secure one (Exercise 14.10, page 609).

Other types of attacks on digital signatures include message confusion and signer confusion. These methods introduce confusion regarding what was signed or who signed it, respectively, by making a single signature valid for two different messages (the authentic one and a forgery) or by allowing the attacker to generate a new public key that is valid for a previously created pair $(M,\Sigma)$, thereby causing confusion about who signed the message. Signatures that resist these two attacks are known as strongly binding signatures. Exercise 13.5 on page 576 shows that any signature scheme can be modified to become this more secure type of signature. The process is as follows.

Given a signature of the form $(G,S,V)$, where $G$ is the key generation process, $S$ is the signing method, and $V$ is the signature verification method, and given a hash function $H'$, we define a new signature $(G,S',V')$ where

$$S'(s_k,m):=\left\{h\gets H'(p_k,m),\sigma\stackrel{R}{\longleftarrow}S(s_k,m),\text{ return }(h,\sigma) \right\}$$
$$V'(p_k,m,(h,\sigma))\text{ accepted if }h=H'(p_k,m)\text{ y }V(p_k,m,\sigma)\text{ is accepted }$$

where $(p_k, s_k)$ is the public-private key pair for the signature. If $S$ is secure, the proof shows that $S'$ is also secure, and if $H'$ is collision-resistant, then $S'$ is strongly binding and therefore resistant to message confusion and signer confusion.

\subsection{Extending the Message Size Using a Hash}

To improve computational efficiency, and as described in \cite{boneh2020graduate}, in Chapter 13.2 (Extending the message space with collision-resistant hashing), we will not sign a message but rather a hash of the message, thereby improving the algorithm’s efficiency. To do this, we will fix the length of our message space $\text{M}$ and, given our signature $\text{S}=(G,S,V)$ defined over a message and signature space $(\text{M},\Sigma)$, and $H:\text{M'}\to\text{M}$ a hash function, where $\text{M'}$ is much larger than $\text{M}$, we make the following trivial modification to develop a new signature scheme $\text{S'}$:

$$S'(s_k,m):=S\left(s_k,H(m)\right)\text{ and }V'(p_k,m,\sigma):=V(p_k,H(m),\sigma)$$

This simple modification allows us to sign very long messages and to set the parameters of our signature to a value $n$ as the fixed initial message length. In this way, we do not need to worry about the initial message length when defining the final values of our signature’s security parameters. In particular, we can consider signing relatively short messages, specifically those of $256$ bits or digits. The following theorem, which appears in \cite{boneh2020graduate} on page 533, states that this procedure does not make the original signature any less secure.

\begin{thm}
    Suppose that the signature scheme $S$ is secure and that the hash function $H$ is collision-resistant. Then the associated signature $\text{S'}=(G,S',V')$, as defined above, is secure.
\end{thm}

\subsection{Compatibility with Padding Schemes}

From a functional perspective, this digital signature can be understood in its signing process as a function, which we have previously denoted as $S$, of the form $S:\text{M}\times Var\to\Sigma$, where $Var$ is the space of variables of the signature. Designed in this way, it is therefore amenable to the application of padding schemes, which have not been included in the design shown here so as not to obscure the explanation of the procedure. In our case, the length of the elements of $\text{M}$ matches that of $\Sigma$, and therefore padding schemes designed for encryption methods can also be used for this digital signature, such as PSS \cite{bellare1998pss} or OAEP+ \cite{shoup2001oaep}.

The underlying problem is that these methods were developed with older signature systems such as RSA in mind, and while their usefulness in improving or strengthening this signature is likely, it has not yet been formally proven; but it is also true that an additional padding scheme does not add a significant computational burden to the signing process, and it is always preferable to proceed with caution, given that changes are much more difficult to make after the fact than beforehand in a digital security scheme, that is, before it has been implemented.

\section{Selection of Digital Signature Parameters}

We will now combine all the previous results to define the parameters of the digital signature. To do so, we must remember that we need to make choices regarding the number and size of the matrices $n$, the number of variables $m$, the size of the finite field $q$, and the desired range $r$. We highlight the minimum necessary conditions that must all be met.

\begin{enumerate}
    \item Avoid variable collisions: $\boldsymbol{q^m\gg 0}$.
    \item There is no $n\times n$ discrete Fourier transform in $\mathbb{F}_q$: $\boldsymbol{n\nmid (q-1)}$.
    \item The probability of obtaining a circulant matrix on the first try is close to $(q-1)/q$: $\boldsymbol{n}$ \textbf{is prime,} $\boldsymbol{\gcd(n,q)=1}$, and $\boldsymbol{ord_n(q)=n-1}$.
    \item The value of $r$ makes cryptanalysis more difficult by increasing the number of cases to consider: $r$ should be as close as possible to $\lceil{\frac{n}{2}}\rceil$.
    \item The probability that the associated MinRank problem is well-defined is as high as possible: $\boldsymbol{r=n-\sqrt{n-1}}$, which thus constrains the value of $r$ and requires that $n$ be a prime number of the form $\boldsymbol{n=u^2+1}$, that is, the number immediately following a perfect square. Landau’s fourth problem, which remains unsolved, posits that there are infinitely many such prime numbers \cite{hardy1979introduction}. This problem is, in fact, a special case of Schinzel’s Hypothesis H, which encompasses many others, such as the existence of infinitely many twin primes \cite{schinzel1958certaines}.
    \item A direct matrix factorization approach is intractable: $$\boldsymbol{n^2(n^2+n-n\sqrt{n-1}+2m+2)^2-n^3(m+1)(m+2)>0}$$
    \item A direct matrix factorization attack after the $m$ variables have been chosen is intractable: $n>1$ and $r=n-\sqrt{n-1}$.
    \item All the attempts to achieve existential forgery fail: $\boldsymbol{(n-1)\nmid (n+m),}$ $\boldsymbol{(n-1)\nmid m,}$ $\boldsymbol{-3+\sqrt{8n+1}<2m<n^2-3n-2}$
\end{enumerate}

\section{Example of Variables}

Let’s take $m=256$ and assume we are working in the finite field with $q=2$, that is, $\mathbb{F}_q=\mathbb{Z}_2$. We have that $q^m\gg 0$. Since $n$ must be prime (and therefore odd, since it must be large), we automatically have that $\gcd(n,q)=1$ and that $n\nmid (q-1)$. For the inequalities in sections 6 and 8 to be satisfied simultaneously, $n\in[34;33152]$ must hold, so we must find an $n$ large enough within that interval.

We must now require that $n$ be a prime of the form $n=u^2+1$, that $ord_n(2)=n-1$, and that $(n-1)\nmid 256$ and $(n-1)\nmid (n+256)$. The prime numbers of the form $n=u^2+1$ (near-square primes) that lie in the interval $[34;33152]$ belong to the set

\begin{gather}    NSP(33,33152)=\{37,101,197,257,401,577,677,1297,1601,2917,3137,4357, \nonumber \\     5477,7057,8101,8837,12101,13457,14401,15377,15877,16901,17957,21317, \nonumber \\    22501,24337,25601,28901,30977,32401\}. \nonumber 
\end{gather}

Of all these values, those that satisfy the condition $ord_n(2)=n-1$ are

$$\{37,101,197,677,4357,5477,8837,16901,17957,21317,22501\}$$

 and all of them satisfy $(n-1)\nmid 256$, $(n-1)\nmid (n+256)$. Therefore, all of them meet the security standards described above, and our only criterion when choosing one value over another would be the size of the keys and the balance between the speed of the procedure and the difficulty of solving the MinRank problem by brute force. The rank associated with each of them is $r=n-\sqrt{n-1}$, which is always a positive integer.

Let’s now consider a different example, and suppose we want a much larger finite field, such as $\mathbb{F}_{256}$. This finite field is the one used in the AES scheme, as described in \cite{DaemenRijmen2002}, and naturally associates each element of a finite field with a byte, since $256=2^8$.

We have $q^m=2^{8m}$, which will be sufficiently large even for relatively small values of $m$. $n$ must satisfy $n\nmid 255$, and since it must also be a prime number, this implies that it must be distinct from $3$, $5$, and $17$, which are the prime factors of $255$. Furthermore, since $\gcd(n,256)=1$, it suffices for $n$ to be odd.

It is easy to show that for the inequality in Section 8 to hold, it suffices that $$n\in\left(\frac{3+\sqrt{17+8m}}{2},\frac{m^2}{2}+\frac{3m}{2}+1\right)$$

If we take, for example, $m=20$, then $n\in(8,231)$. Of these, the primes of the form $n=u^2+1$ are $$n\in\{17,37,101,197\}$$

It is also easy to verify that the inequality in Section 6 is always true for all $n$ if $m\leq 57$. From the above values, we must exclude $17$ because it is a divisor of $255$, and $37$ because $255$ is not a primitive root in $\mathbb{Z}_{37}$. In the case of $101$ and $197$, $255$ is indeed a primitive root in $\mathbb{Z}_{101}$ and $\mathbb{Z}_{197}$. Since both values also satisfy $(n-1)\nmid m$ and $(n-1)\nmid(m+n)$, our two pairs of parameters are $q=256, m=20, n=101, r=91$ and $q=256, m=20, n=197, r=183$.

\bibliographystyle{plain}
\bibliography{bibliography}

\end{document}